\documentclass[10pt]{amsart}
\usepackage{amsfonts,amsmath,amscd,amssymb,amsthm,mathrsfs,esint}
\usepackage{color,fancyhdr,framed,graphicx,verbatim,mathtools}

\usepackage[margin=1in]{geometry}

\usepackage{tikz}
\usetikzlibrary{matrix,arrows,decorations.pathmorphing}

\numberwithin{equation}{section}

\newtheorem{theorem}{Theorem}[section]
\theoremstyle{definition}
\newtheorem{corollary}[theorem]{Corollary}
\newtheorem{lemma}[theorem]{Lemma}

\newtheorem{definition}[theorem]{Definition}

\newtheorem{remark}[theorem]{Remark}

\newtheorem*{theorem*}{Theorem}
\newtheorem*{corollary*}{Corollary}
\newtheorem*{remark*}{Remark}
\newtheorem*{lemma*}{Lemma}

\title[Weyl's Law on Compact Heisenberg Manifolds]{A Tauberian Approach to an Analog of Weyl's law for the Kohn Laplacian on Compact Heisenberg Manifolds}

\author{Colin Fan} 
\address[Colin Fan]{Rutgers University--New Brunswick, Department of Mathematics, Piscataway, NJ 08854, USA}
\email{colin.fan@rutgers.edu}

\author{Elena Kim}
\address[Elena Kim]{Massachusetts Institute of Technology, Department of Mathematics, Cambridge, MA 02142, USA}
\email{elenakim@mit.edu}

\author{Yunus E. Zeytuncu}
\address[Yunus E. Zeytuncu]{University of Michigan--Dearborn, Department of Mathematics and Statistics, 
Dearborn, MI 48128, USA}
\email{zeytuncu@umich.edu}

\thanks{This work is supported by NSF (DMS-1950102 and DMS-1659203). The work of the last author is also partially supported by a grant from the Simons Foundation (\#353525).}


\begin{document}
\begin{abstract}
Let $M= \Gamma \setminus \mathbb{H}_d$ be a compact quotient of the $d$-dimensional Heisenberg group $\mathbb{H}_d$ by a lattice subgroup $\Gamma$. We show that the eigenvalue counting function $N(\lambda)$ for any fixed element of a family of second order differential operators $\left\{\mathcal{L}_\alpha\right\}$ on $M$
has asymptotic behavior $N\left(\lambda\right) \sim C_{d,\alpha} \operatorname{vol}\left(M\right) \lambda^{d + 1}$, where $C_{d,\alpha}$ is a constant that only depends on the dimension $d$ and the parameter $\alpha$. As a consequence, we obtain an analog of Weyl's law (both on functions and forms) for the Kohn Laplacian on $M$. Our main tools are Folland's description of the spectrum of $\mathcal{L}_{\alpha}$ and Karamata's Tauberian theorem.
\end{abstract}

\maketitle

\section{Introduction}

\subsection{Motivation}
Motivated by the celebrated Weyl's law, we aim to study the asymptotic behavior of eigenvalues of the Kohn Laplacian $\square_b$ (also referred to as the complex Laplacian) on compact Heisenberg manifolds, specifically compact quotients of the Heisenberg group by lattice subgroups. Much of our work is inspired by \cite{REU2020Weyl}, where the authors compute the leading coefficient of the eigenvalue counting function for $\square_b$ on functions on the $(2n-1)$-dimensional sphere $S^{2n-1}$. As in the original Weyl's law, here the leading coefficient is proportional to the volume of $S^{2n-1}$, multiplied by a constant that depends only on the dimension $n$. This constant is expressed as an integral and is similar to the constant that appears in \cite[Theorem 6.1]{Stanton1984TheHE}. Note that the result in \cite{Stanton1984TheHE} examines Weyl's law for the Kohn Laplacian on $\left(p,q\right)$-forms, where $0 < q < n - 1$, on compact strongly pseudoconvex embedded CR manifolds of hypersurface type in $\mathbb{C}^n$ for $n\geq 3$. A similar analog of Weyl's law for the Kohn Laplacian on functions on such general CR manifolds is an open problem. \cite{REU2020Weyl} gives an answer on spheres.

In this paper, we obtain an analog of Weyl's law for the Kohn Laplacian on functions and on differential forms on compact Heisenberg manifolds.
We first note that the Heisenberg group has two distinguished left-invariant differential operators: $\mathcal{L}_0$ and $i^{-1} T$. For $\alpha \in \mathbb{R}$, the family of left-invariant operators given by $\mathcal{L}_\alpha = \mathcal{L}_0 +  i \alpha T$ is also of importance due to its relation to the Kohn Laplacian. In fact, the spectral analysis of $\square_b$ reduces to understanding $\mathcal{L}_\alpha$. 
We note that every positive real number is an eigenvalue of $\mathcal{L}_\alpha$ on the Heisenberg group, see \cite{STRICHARTZ1991350}, therefore the spectrum is not discrete. Thus, it is not a suitable manifold on which to count eigenvalues. However, on compact quotients $M$, the operators $\mathcal{L}_\alpha$ have discrete spectra, as noted in \cite{Folland2004CompactHM}. Thus, we can count the eigenvalues on these compact Heisenberg manifolds. We note that obtaining the asymptotics of a counting function for a given positive sequence of numbers is not always straightforward. 

In \cite{Strichartz2015} and \cite{Taylor1986}, the authors study the distribution of eigenvalues on compact quotients for the single operator $\mathcal{L}_0$. In particular, they obtain
the asymptotic result $N \left(\lambda\right)\sim C_{d,0} \operatorname{vol} \left(M\right) \lambda^{ d + 1}$, where $N(\lambda)$ is the eigenvalue counting function for $\mathcal{L}_0$ and $C_{d,0}$ is a constant that that depends only on the dimension $d$. In \cite{Strichartz2015}, Strichartz obtains his result by using Folland's explicit spectrum for $\mathcal{L}_0$ and a careful analysis of the asymptotics of binomial coefficients.
On the other hand, in \cite{Taylor1986}, Taylor uses asymptotics of the trace of the heat kernel and Karamata's Tauberian theorem, with no reference to an explicit spectrum. In this note, by combining both the explicit spectrum from \cite{Folland2004CompactHM} and Karamata's Tauberian theorem, we obtain asymptotics for a family of second order differential operators $\mathcal{L}_\alpha$, for $-d \leq \alpha \leq d$.


As a corollary, we obtain an analog of Weyl's law for the Kohn Laplacian on functions and differential forms on $M$. We note that our result on $\left(p,q\right)$-forms, up to a simple dimensional constant, matches the Weyl's law analog in \cite{Stanton1984TheHE}. Furthermore, the Weyl's law analog we obtain for functions matches, up to the same dimensional constant as before, with the Weyl's law for spheres in \cite{REU2020Weyl}. These observations provide more insight into the open problem mentioned above.


\subsection{Preliminaries}
We follow the exposition of the Heisenberg group and the Kohn Laplacian in \cite{Folland2004CompactHM} closely and refer the reader to that paper. We also refer the reader to \cite{CanarecciMasterthesis} and \cite[Chapter XIII]{Stein}  for further definitions and details, and \cite{CS01} for a detailed introduction to the Kohn Laplacian on CR manifolds.

\begin{definition} \label{def:Heisenberg}
The $d$-dimensional {\it Heisenberg group}, $\mathbb{H}_d$, is the set $\mathbb{C}^d \times \mathbb{R}$ along with the group law defined by
\[\left(z,t\right) \cdot \left(z',t'\right) = \left(z + z', t + t' + 2\operatorname{Im}\left\langle z, z'\right\rangle\right),\]
where $z,z' \in \mathbb{C}^d$; $t, t'\in \mathbb{R}$; and $\left\langle z, z'\right\rangle = z_1 \overline{z}'_1 + \cdots + z_d \overline{z}'_d$.
\end{definition}

Note that $\mathbb{H}_d$ embeds naturally in $\mathbb{C}^{d+1}$ under the identification 
\[\left(z,t\right) \mapsto \left(z,t + i\left|z\right|^2\right)\]
and therefore it is an embedded CR manifold of hypersurface type.

The Heisenberg group can be alternatively described in polarized coordinates. That is, $\mathbb{H}_d$ is the set $\mathbb{R}^d \times \mathbb{R}^d \times \mathbb{R}$ with the group law
\[\left(p,q,s\right) \cdot \left(p',q',s'\right) = \left(p + p', q + q', s + s' + p\cdot q'\right).\]

 For $-d \leq \alpha \leq d$, define the second order differential operator
\[\mathcal{L}_a = -\frac{ 1}{2} \sum_{j=1}^d \left( Z_j \overline{Z}_j +
    \overline{Z}_j Z_j\right) + i \alpha T,\]
where
\[\overline{Z}_j = \frac{\partial }{\partial \overline{z}_j} - i
      z_j \frac{\partial }{\partial t} \quad \text{ and } \quad T =
      \frac{\partial }{\partial t}.\]

The following properties of $\mathcal{L}_\alpha$ and $\square_b$ are well-known and documented in \cite{Folland2004CompactHM}. To investigate the spectral asymptotics of $\mathcal{L}_\alpha$, it is convenient to study $\mathcal{L}_0$ and $i^{-1}T$ separately as they are essentially self-adjoint strongly commuting operators. The connection between $\mathcal{L}_\alpha$ and $\square_b$ is then given by the diagonal action of $\square_b$ on $\left(0,q\right)$-forms, $0 \leq q \leq d$:
\[\square_b\left( \sum_{\left|J\right| = q} f_J d \overline{z}^J\right) = \sum_{\left|J\right| = q} \mathcal{L}_{d - 2q} f_J d \overline{z}^J,\]
where $f_J$ are functions, $J=\left(j_1, \ldots, j_q\right)$ with $1 \leq j_1  < \cdots <j_q \leq d$, and $d\overline{z}^J=d\overline{z}_1 \wedge \cdots \wedge d\overline{z}_q$.

Let $\Gamma$ be a lattice subgroup of $\mathbb{H}_d$, that is, a discrete subgroup so that $M = \Gamma\setminus \mathbb{H}_d$ is a compact manifold. Note that the CR structure and the operators $\mathcal{L}_\alpha$ and $T$ descend onto $M$. This makes $M$ a strongly pseudoconvex CR manifold. Furthermore, the diagonal action of $\square_b$ descends onto $M$.

Importantly, for all $\alpha$, $\mathcal{L}_\alpha$ on $M$ has discrete eigenvalues which are explicitly given in \cite{Folland2004CompactHM}. To obtain an analog of Weyl's law, we use a  generating functions argument and invoke
Karamata's Tauberian theorem (see \cite[Theorem 1.1, page 57]{ANPS09}).

\begin{theorem*}[Karamata]\label{thm:tauberian}
Let $\left\{\lambda_j\right\}_{j \in \mathbb{N}}$ be a sequence of positive real numbers such that $\sum_{j \in \mathbb{N}}e^{-\lambda_j t}$ converges for every $t>0$. Then for $n>0$ and $a \in \mathbb{R}$, the following are equivalent.
\begin{enumerate}
    \item $\lim_{t \to 0^+} t^n\sum_{j \in \mathbb{N}}e^{-\lambda_j t} =a$
    \item $\lim_{\lambda \to \infty} \frac{N\left(\lambda\right)}{\lambda^n}=\frac{a}{\Gamma\left(n+1\right)}$
\end{enumerate}
where $N\left(\lambda\right)=\#\left\{j:\lambda_j\leq \lambda\right\}$ is the counting function. 
\end{theorem*}

\subsection{Main Results}
Putting these ideas together yields the following analog of Weyl's law for $\mathcal{L}_\alpha$. 

\begin{theorem}\label{mainT}
 Let $N(\lambda)$ be the eigenvalue counting function for $\mathcal{L}_\alpha$ on $L^2\left(M\right)$ for $-d \leq \alpha \leq d$. 
   For $-d < \alpha < d$,
    \[\lim_{\lambda\to\infty} \frac{N \left(\lambda\right)}{\lambda^{d + 1}} =\operatorname{vol}\left(M\right) \frac{2}{\pi^{d + 1}\Gamma\left(d + 2\right)} \int_{-\infty}^\infty \left(\frac{x}{\sinh x}\right)^d e^{-\alpha x}\,dx\]
    and for $\alpha = \pm d$,
    \[\lim_{\lambda\to\infty} \frac{N \left(\lambda\right)}{\lambda^{d + 1}} =\operatorname{vol}\left(M\right) \frac{2d}{\left(d+1\right)\pi^{d + 1}\Gamma\left(d+2\right)} \int_{-\infty}^\infty \left(\frac{x}{\sinh x}\right)^{d + 1} e^{-\left(d - 1\right)x }\,dx.\]
\end{theorem}

From this statement and the diagonal action of $\square_b$ we can obtain the following statement on $\left(p,q\right)$-forms. Note that we require $d\geq 2$ to obtain nontrivial $\left(p,q\right)$-forms.

\begin{corollary}\label{mainC}
Fix $d \geq 2$. Let $N(\lambda)$ be the eigenvalue counting function for $\square_b$ on $M$ acting on $\left(p,q\right)$-forms, where $0 \leq p < d + 1$ and $0 < q < d$. We have that
\[\lim_{\lambda\to\infty} \frac{N \left(\lambda\right)}{\lambda^{d + 1}} = \operatorname{vol} \left(M\right) \binom{d}{p}\binom{d}{q} \frac{2}{\pi^{d + 1} \Gamma \left( d + 2\right)} \int_{-\infty}^\infty \left( \frac{x}{\sinh x}\right)^d e^{- \left(d - 2q\right)x}\,dx.\]
\end{corollary}


In the remainder of the paper, we prove the main theorem and its corollary.

\section{Proofs}

\subsection{Compact Quotients} For the proof of our theorem, we first establish some notation and properties of lattice subgroups of $\mathbb{H}_d$.

\begin{definition}
Let $\ell = \left(\ell_1,\ell_2,\ldots,\ell_d\right)$ be a $d$-tuple of positive integers so that $\ell_1 \mid \ell_2 \mid \cdots \mid \ell_d$. Define
\[\Gamma_\ell = \left\{\left(p,q,s\right): p,q\in \mathbb{Z}^d, s\in \mathbb{Z}, \ell_j \mid q_j\text{ for all } 1 \leq j \leq d\right\}.\]
$\Gamma_\ell$ is a lattice subgroup of the polarized Heisenberg group. Importantly, for a given $\Gamma_\ell$, define the constant $L = \ell_1 \ell_2 \cdots \ell_d$.
\end{definition}

\begin{theorem*}
\cite[Proposition 2.1]{Folland2004CompactHM} Given any lattice subgroup $\Gamma$ of $\mathbb{H}_d$, there exists a single $\ell$ and an automorphism $\Phi$ of $\mathbb{H}_d$ so that $\Phi \left(\Gamma\right) = \Gamma_\ell$. 
\end{theorem*}

Thus, we can associate every lattice subgroup with the constant $L$ given by $\Gamma_\ell$. 
Another useful property is that the center of a lattice subgroup is of the form $\left(0,0,c\mathbb{Z}\right)$ for some $c>0$. Knowing this information, the volume of $M$ can be computed as $L c^{d + 1}$, as shown in \cite{Strichartz2015}. We can now state Folland's result on the joint spectrum of $\mathcal{L}_0$ and $i^{-1} T$.

\begin{definition}
Let $A$ and $B$ be two operators on a vector space $V$. The {\it joint spectrum} of $A$ and $B$ is
\[\sigma \left(A,B\right) = \left\{\left(\lambda,\mu\right): v\in V\setminus\left\{0\right\}, Av = \lambda v, Bv = \mu v\right\},\]
counting multiplicities of $\left(\lambda,\mu\right)$.
\end{definition}

\begin{theorem*}
\cite[Theorem 3.2]{Folland2004CompactHM} Given a lattice with center $\left(0,0,c\mathbb{Z}\right)$, the joint
spectrum of $\mathcal{L}_0$ and $i ^{-1} T$ on $L^2 \left( M\right)$ is
\[\left\{ \left( \frac{ \pi \left| n \right|}{2c} \left( d + 2j
      \right), \frac{ \pi n }{2c} \right): j \in \mathbb{Z}_{\geq
      0}, n \in \mathbb{Z}\setminus \left\{ 0 \right\}\right\}
  \cup \left\{ \left( \frac{ \pi}{2} \left| \xi \right|^2,0
    \right): \xi \in \Lambda'\right\}\]
and the multiplicity of $ \left( \frac{ \pi \left| n \right|}{2c}
  \left( d + 2j \right), \frac{ \pi n }{2c} \right)$ is
\[\left| n\right|^d L \binom{j + d -
    1}{d - 1}.\]
    
$\Lambda'$ is the dual lattice of the lattice $\Lambda = \pi \left(\Gamma\right)$, where $\pi : \mathbb{H}_d \to \mathbb{C}^d$ is the quotient map $\pi\left(z,t\right) = z$. The multiplicity of an eigenvalue coming from the second set is dependent on the structure of $\Lambda'$. 
\end{theorem*}

Since
$\mathcal{L}_0$ and $i ^{-1} T$ are self-adjoint strongly
commuting operators, as shown in \cite{STRICHARTZ1991350}, we have the following corollary of the theorem above.
\begin{corollary*}\cite[Corollary 3.3]{Folland2004CompactHM}
For $\alpha\in \mathbb{R}$, the spectrum of $\mathcal{L}_\alpha$ on $M$ is
\[ \underbrace{\left\{ \frac{ \pi \left| n \right|}{2c} \left( d + 2j - \alpha
    \operatorname{sgn} n
  \right): j \in \mathbb{Z}_{\geq 0}, n \in \mathbb{Z}\setminus
  \left\{ 0 \right\}\right\}}_{\text{type } \left(a\right)} \cup \underbrace{\left\{ \frac{ \pi}{2} \left|
    \xi \right|^2: \xi \in \Lambda'\right\}}_{\text{type } \left(b\right)}.\]
\end{corollary*}
We label the
eigenvalues in the first set as type $\left(a\right)$, and the eigenvalues in
the second set as type $\left(b\right)$. Moreover, since the operators are self-adjoint and strongly commuting, the total multiplicity of a type $\left(a\right)$ eigenvalue $\lambda$ is the sum of the multiplicities of elements from the joint spectrum coming from distinct $n,j,\xi$ that add up to $\lambda$. For example, if
\[\lambda = \frac{\pi \left|n\right|}{2c} \left(d + 2j - \alpha \operatorname{sgn} n\right) = \frac{\pi \left|n'\right|}{2c} \left(d + 2j' - \alpha \operatorname{sgn} n'\right) = \frac{\pi}{2}\left|\xi\right|^2\]
for some $\xi$ and  $\left(n,j\right) \neq \left(n',j'\right)$,
then the multiplicity of $\lambda$ is exactly
\[\left|n\right|^d L \binom{j + d - 1}{d - 1} + \left|n'\right|{^d} L \binom{j' + d - 1}{d - 1} + \operatorname{mult} \left(\frac{\pi}{2} \left|\xi\right|^2\right).\]
The corollary above allows us to define the generating function $\sum_{\lambda} e^{-\lambda_j t}$ where the terms are repeated according to the multiplicity of $\lambda_j$. This function appears in section \ref{sect:proof_of_main}, where we invoke a Tauberian theorem to understand the distribution of eigenvalues.

We now decompose the eigenvalue counting function $N\left(\lambda\right)$ for $\mathcal{L}_\alpha$ into two parts. Let $N_a\left(\lambda\right)$ and $N_b\left(\lambda\right)$ be the positive
eigenvalue counting functions of type $\left(a\right)$ and $\left(b\right)$ respectively for
$\mathcal{L}_\alpha$. Formally, \[N_a\left(\lambda\right)=\#\left\{j: 0< \lambda_j\leq \lambda, \lambda_j \text{ is of type }\left(a\right)\right\} \text{ and }N_b\left(\lambda\right)=\#\left\{j: 0<\lambda_j\leq \lambda, \lambda_j \text{ is of type }\left(b\right)\right\}.\]
Therefore, to study $N(\lambda)$, it suffices to analyze $N_a(\lambda)$ and $N_b(\lambda)$ separately. 

\begin{remark}
Finally, before we provide the details of the proofs, we make a note on isospectral quotients. As noted in \cite{Folland2004CompactHM} the automorphisms of the Heisenberg group decompose  into three categories: symplectic automorphisms, inner automorphisms, and dilations. We suspect that few automorphisms of $\mathbb{H}_d$ yield isospectral quotient manifolds. That is, if $\varphi\in \operatorname{Aut}\left(\mathbb{H}_d\right)$, then $\Gamma\setminus \mathbb{H}_d$ and $\varphi\left(\Gamma\right)\setminus \mathbb{H}_d$ are unlikely to be isospectral. Our reasoning is based on the following observations. Indeed, dilations by $r$ change both type $\left(a\right)$ and $\left(b\right)$ eigenvalues by a factor of $r^2$. Similarly, we see that inner automorphisms by $\left(w,t\right)$, though they preserve the lattice structure, are unlikely to preserve the center for generic $w$. Thus, the symplectic matrices that preserve the lengths and multiplicities in the dual lattice are the only building block automorphisms of $\mathbb{H}_d$ that can reasonably yield isospectral manifolds. However, such a statement does not yield a rich class of examples. We leave a formal statement and investigation of isospectral Heisenberg manifolds to another study.
\end{remark}

\subsection{Sums to Integrals} In this part, we provide the analytical details of the proofs.
Taking a cue from \cite{REU2020Weyl}, we first define the scaled ceiling function.

\begin{definition}\label{def:scaledceil}
For $t> 0$, the {\it scaled ceiling function} $\left\lceil \cdot \right\rceil_t:\mathbb{R}\rightarrow \mathbb{R}$ is 
\[\left\lceil x \right\rceil_t =t \left\lceil x/t \right\rceil.\]
\end{definition}
Note that 
\[\left\lceil x \right\rceil_t=t \min\left\{n \in \mathbb{Z} : n \geq x/t\right\}=t \min\left\{n \in \mathbb{Z} : tn \geq x\right\}.\]
Therefore, $\left\lceil x \right\rceil_t$ can be thought of $x$ rounded up to the nearest integer multiple of $t$. This implies that for a fixed $x \in \mathbb{R}$ and $t>0$, we have $0 \leq \left\lceil x \right\rceil_t  -x < t $. As direct consequence, we have the following two properties. 

\begin{enumerate}
    \item \label{property1} For any fixed $x \in \mathbb{R}$, $\lim_{t \to 0^+}\left \lceil x \right\rceil_t =x$.
    \item \label{property2} Let $f:\left[a,b\right] \to \mathbb{R}$ be a monotonically decreasing function. Then for a fixed $0< t<b-a$, for all $x \in \left[a,b-t\right]$, we have $f\left(\left\lceil x \right\rceil_t\right) \leq f\left(x\right)$.
\end{enumerate}

The following lemma makes use of the definition of the scaled ceiling function to convert a right Riemann sum into an integral. It is used to simplify calculations in the proof of our main theorem. 

\begin{lemma}\label{lem:convertintegral}
For $u,v>0$,
\[t^{d+1} \sum_{n=1}^\infty n^d \frac{ e^{-t uv n }}{\left( 1 -e^{-2t u n} \right){^d}}=\int_{0}^\infty  \left\lceil x \right\rceil_t^d \frac{ e^{- u v\left\lceil x \right\rceil_t  }}{\left( 1 -e^{-2 u \left\lceil x \right\rceil_t} \right){^d}}\, dx.\]
\end{lemma}

\begin{proof}
We have that
\begin{align*}
t^{d+1} \sum_{n=1}^\infty n^d \frac{ e^{-t uv n  }}{\left( 1 -e^{-2t u n} \right){^d}} 
&= t^{d + 1} \sum_{n=1}^\infty \int_{n-1}^n  \left\lceil m \right\rceil^d \frac{ e^{-t u v\left\lceil m \right\rceil }}{\left( 1 -e^{-2t u \left\lceil m \right\rceil} \right){^d}}\, dm\\
&= \int_{0}^\infty t^{d+1} \left\lceil m \right\rceil^d \frac{ e^{-t u v\left\lceil m \right\rceil}}{\left( 1 -e^{-2t u \left\lceil m \right\rceil} \right){^d}}\, dm\\
&= \int_{0}^\infty \left(t \lceil x/t \rceil\right)^d  \frac{ e^{-t u v\left\lceil x/t \right\rceil  }}{\left( 1 -e^{-2t u \lceil x/t \rceil} \right){^d}}\, dx\\
&= \int_{0}^\infty  \left\lceil x \right\rceil_t^d \frac{ e^{- u v\left\lceil x \right\rceil_t}}{\left( 1 -e^{-2 u \left\lceil x \right\rceil_t} \right){^d}}\, dx,
\end{align*}
thus completing the proof.
\end{proof}

The next lemma demonstrates how the scaled ceiling function can be removed from the integrand through a limit. 

\begin{lemma} \label{lem:exchange_lim_int}
For $u, v > 0 $,
\begin{equation*}\label{eq:int1}
\lim_{t \to 0^+} \int_{0}^{\infty} \left\lceil x \right\rceil_{t}^d \frac{e^{-uv\left\lceil x \right\rceil_{t}}}{\left(1-e^{-2u \left\lceil x \right\rceil_{t}}\right){^d}}\,dx =\int_{0}^{\infty} x^d \frac{e^{-u v x}}{\left(1-e^{-2u x}\right){^d}}\, dx.
\end{equation*}
\end{lemma}
\begin{proof}
Define 
\[f\left(x\right) = x^d \frac{ e^{- u v x}}{\left( 1 - e^{-2u x}\right){^d}}.\]
Note that there exists an $M>0$ such that for all $x\geq M$, we have
\[\frac{1}{2} < \left(1 - e^{-2u x}\right){^d} \text{ and } x^d \leq e^{cx} \text{, where }  c = \frac{uv}{2}.\]
It follows that for all $x\geq M$,
\[f\left(x\right) \leq 2 e^{- cx}.\]

By compactness, continuity, and property (2) of the scaled ceiling function, it follows that for all $n\in \mathbb{N}$, $f(\left \lceil x\right\rceil_{1/n})$ is dominated by $R\chi_{\left[0,M\right]} + 2 e^{- cx}$ for some $R > 0$. Thus, we can apply the dominated convergence theorem and use property (1) of the scaled ceiling function to obtain the claim.
\end{proof}

When we invoke the above two lemmas in the following section, we only consider $v = d\pm \alpha$, where $- d < \alpha < d$.

\subsection{Proof of Theorem \ref{mainT}} \label{sect:proof_of_main}
In this section, we prove an asymptotic result for $N_a(\lambda)$, from which Theorem \ref{mainT} follows.

\begin{theorem}\label{thm:main}
  Fix $-d \leq \alpha \leq d$. We have that
  \[\lim_{\lambda\to\infty} \frac{N_a \left(\lambda\right)}{\lambda^{d + 1}} =C_{d,\alpha} \operatorname{vol} \left(M\right) . \] 
\end{theorem}

\begin{proof}
By symmetry of $\operatorname{sgn} n$, it suffices to consider the case where $0 \leq \alpha \leq d$. We separate the cases for $0 \leq
\alpha < d$ and $\alpha = d$ and center our approach on Karamata's
Tauberian theorem. Let $u = \frac{\pi}{2c}$.

  First assume $0 \leq \alpha < d$. Setting $G(t)=\sum_{j \in \mathbb{N}} e^{-\lambda_j t}$, where $\lambda_j$ are the type $\left(a\right)$ eigenvalues of $\mathcal{L}_\alpha$ on $M$ included with multiplicity, we see that
  \begin{align*}
    G \left( t\right)
    &=
    \sum_{\substack{n \in \mathbb{Z}\setminus \left\{ 0\right\} \\
    j \in \mathbb{Z}_{\geq 0}}} \left| n\right|^d L \binom{j + d -
    1}{d - 1} e ^{- t u \left| n \right|\left( d + 2j - \alpha
    \operatorname{sgn} n \right)}\\
    &=L \sum_{\substack{n = 1\\ j = 0}}^\infty n^d \binom{j + d -
    1}{d - 
    1} e ^{- t u n \left( d + \alpha + 2j \right)} + L
    \sum_{\substack{n = 1\\ j = 0}}^\infty n^d \binom{j + d - 1}{d
    - 1} e 
    ^{- t u n \left( d - \alpha + 2j \right)}\\
    &= L \left(G_{-} \left( t\right) +  G_+ \left( t\right)\right),
  \end{align*}
  where $G_-,G_+$ are the parts of $G$, excluding multiplication by $L$, indexed by negative and positive $n$ respectively.
  Recall that
  \[\frac{ 1}{\left( 1 - z \right)^d} = \sum_{j=0}^\infty \binom{j
      + d - 1}{d - 1} z^j.\]
  We see that,
  \begin{align*}
    G_- \left( t\right)
    &= \sum_{n=1}^\infty n^d e^{-t u n \left( d + \alpha \right) }
      \sum_{j=0}^\infty \binom{j + d - 1}{d - 1} e^{-2t u n j} = \sum_{n=1}^\infty n^d
      \frac{ e^{-t u n \left( d + \alpha \right) }}{\left( 1 -
      e^{-2t u n} \right){^d}} \\
          \intertext{and}
    G_+ \left( t\right)
    &= \sum_{n=1}^\infty n^d e ^{-t u n \left( d - \alpha \right)}
      \sum_{j=0}^\infty \binom{j + d - 1}{d - 1}  e ^{-2 t u n j} =
      \sum _{n = 1} ^{\infty} n ^{d}
      \frac{ e^{-t u n \left( d - \alpha \right) }}{\left( 1 -
      e^{-2t u n} \right){^d}}.
  \end{align*}
  To analyze $t^{d+1} G \left( t\right)$, we convert the above sums
  into integrals. We have that
  
 \begin{align*}
 \lim _{t\to 0^+}t ^{d + 1} G \left( t\right) &=\lim_{t \rightarrow 0^+}L\left(\sum _{n = 1} ^{\infty} n ^{d} t^{d+1}\frac{ e^{-t u n \left( d + \alpha \right) }}{\left( 1 -  e^{-2t u n} \right){^d}}+\sum _{n = 1} ^{\infty} n ^{d} t^{d+1}\frac{ e^{-t u n \left( d - \alpha \right) }}{\left( 1 -  e^{-2t u n} \right){^d}} \right)\\
 &=L\lim_{t\rightarrow 0^+}  \left(\int_0^\infty \lceil x \rceil_t^d \frac{ e ^{- u \left( d +\alpha \right)\lceil x \rceil_t}}{\left( 1 - e ^{-2u\lceil x \rceil_t} \right){^d}}\,dx + \int_0^\infty \lceil x \rceil_t^d \frac{ e ^{- u \left( d -\alpha \right)\lceil x \rceil_t}}{\left( 1 - e ^{-2u\lceil x \rceil_t} \right){^d}}\,dx\right) & \text{(Lemma \ref{lem:convertintegral})}\\
 &= L \left(\int_0^\infty x^d \frac{ e ^{- u \left( d +\alpha \right)x}}{\left( 1 - e ^{-2ux} \right){^d}}\,dx + \int_0^\infty x^d \frac{ e ^{- u\left( d -\alpha \right)x}}{\left( 1 - e ^{-2ux} \right){^d}}\,dx\right) &\text{(Lemma \ref{lem:exchange_lim_int})} \\
 &= L \int_{-\infty}^\infty x^d \frac{ e ^{- u \left( d + \alpha \right)x}}{\left( 1 - e ^{-2ux}\right){^d}}\,dx\\
 &= L \int_{-\infty}^\infty x^d \frac{ e ^{- \frac{ \pi}{2c} \left( d + \alpha \right)x}}{\left( 1 - e ^{-\frac{ \pi}{c}x}\right){^d}}\,dx.
 \end{align*}
  Let $v = \frac{\pi x}{2c}$.  Recalling that $\operatorname{vol}(M)=Lc^{d+1}$, we have that
  \[\lim _{t\to 0^+} t^{d+1}G \left( t\right) = \operatorname{vol} 
      \left( M\right)\frac{2^{d + 1}
       }{\pi^{d+1}} \int_{-\infty}^\infty v^d\frac{
      e ^{-  \left( d + 
            \alpha \right)v}}{\left( 1 - e ^{-2v}
        \right){^d}}\,dv =\operatorname{vol} 
      \left( M\right) \frac{2}{\pi^{d + 1}} \int_{-\infty}^\infty \left(\frac{v}{\sinh v}\right)^d e^{- \alpha v} \,dv.\]
  Therefore,
  \[\lim_{\lambda\to\infty} \frac{N_a \left(\lambda\right)}{\lambda^{d + 1}} = \operatorname{vol}\left(M\right) \frac{2}{\pi^{d + 1} \Gamma \left( d + 2\right)}\int_{-\infty}^\infty \left(\frac{x}{\sinh x}\right)^d e^{- \alpha x} \,dx.\]

  Now fix $\alpha = d$. Note that when both $j=0$ and $n > 0$, we obtain
  eigenvalues equal to zero. This case is omitted as the kernel of $\mathcal{L}_d$ is infinite dimensional. Let $u = \pi/c$. We see that
  \begin{align*}
    G \left( t\right)
    &= L \sum _{\substack{n=1 \\ j = 0}}^\infty n^d
    \binom{j + d 
        - 1}{d - 1} e ^{- t u n \left( d + j \right)} + L \sum
        _{\substack{n = 1 \\ j = 1}}^\infty n^d \binom{j + d -
    1}{d - 1} e ^{- t u n j}\\
    &= L\left( G_- \left( t\right) +  G_+ \left( t\right)\right).
  \end{align*}
  Note that the $G_-$ and the $G_+$ that appear here are different from the previous case. By a similar analysis,
  \begin{align*}
    G_- \left( t\right)
    &= \sum _{n = 1} ^{\infty} n^d \frac{ e ^{-t u n d}}{\left( 1
      - e ^{-t u n} \right){^d}}\\
    G_+ \left( t\right)
    &= \sum _{n = 1} ^{\infty} n^d \left( \frac{ 1}{\left( 1 - e
      ^{ -t u n} \right){^d}} - 1\right). 
  \end{align*}
  We now convert to integrals. We refer to the analysis of $G_+$ in  \cite[Lemma 2.5 and Proposition 2.8]{REU2020Weyl} and to the analysis of $G_-$ in \cite[Lemma 2.11 and Proposition 2.13]{REU2020Weyl}. From their calculations, we obtain 
  \[\lim _{t\to 0^+} t^{d + 1} G \left( t\right) = L \int_0^\infty
    x^d \frac{ 1}{\left(e ^{\frac{\pi x}{c}} - 1\right){^d}}\,dx + L
    \int_0^\infty x^d \left( \frac{ 1}{\left( 1 - e ^{\frac{-\pi x}{c}}
        \right){^d}} - 1\right)\,dx.\]
  Let $v = \frac{\pi x}{2c}$.  We have that
  \[\lim _{t\to 0^+} t ^{d + 1} G \left( t\right) =\operatorname{vol} \left( M \right) d! \frac{ 2^{d
      + 1}}{\pi^{d + 1}}
  \frac{ 1}{d!} \int_0^\infty v^d \left( \frac{ 1}{\left( 1 -
        e^{-2 v} \right){^d}} - 1 + \frac{ 1}{\left( e^{2 v} - 1
      \right){^d}}\right)\,dv.\]
  The above integral is manipulated in \cite{REU2020Weyl} to obtain a form that is compatible with the results of \cite{Stanton1984TheHE}. Following their computation,
  \begin{align*}
    \lim _{t\to 0^+} t^{d + 1} G \left( t\right)
    &=   \operatorname{vol}
      \left( M\right) d! \frac{ 2^{d
      + 1}}{\pi^{d + 1}}
  \operatorname{vol} \left( S^{2d + 1}\right) \frac{ d}{\left(
      2\pi \right)^{d + 1} \left( d + 1 \right)}
  \int_{-\infty}^\infty \left( \frac{ x}{\operatorname{sinh}x
      } \right)^{d + 1} e^{-\left(d - 1\right)x }\,dx\\
    &=  \operatorname{vol}
      \left( M\right)\frac{ 2}{\pi^{d + 1}} \frac{ d}{d + 1} \int_{-\infty}^\infty \left( \frac{
      x}{\operatorname{sinh}x } \right)^{d + 1} e^{-\left(d - 1\right)x}\,dx .
  \end{align*}
  Therefore,
  \[\lim_{\lambda\to\infty} \frac{N_a \left(\lambda\right)}{\lambda^{d + 1}} = \operatorname{vol} \left(M\right) \frac{2d}{\left(d + 1\right)\pi^{d + 1}\Gamma \left(d  + 2\right)}\int_{-\infty}^\infty \left( \frac{
      x}{\operatorname{sinh}x } \right)^{d + 1} e^{-\left(d - 1\right)x}\,dx, \]
completing our proof.
\end{proof}

After understanding the asymptotics of $N_a(\lambda)$, we look at the distribution of type $(b)$ eigenvalues, which comes down to counting lattice points in $\mathbb{R}^{2d}$. Here, we use a more general theorem on the distribution of eigenvalues for the standard Laplacian on flat tori. Indeed, we invoke the following statement from \cite[page 26]{ANPS09}.

\begin{theorem*}[Weyl's Law for Flat Tori]
    Let $\Lambda$ be a full-rank lattice in $\mathbb{R}^n$, and $N\left(\lambda\right)$ be the eigenvalue counting function for the standard Laplacian on the flat torus, $T = \Lambda \setminus \mathbb{R}^n$. That is,
    \[N \left(\lambda\right) = \# \left\{\mu \in \Lambda' : \left| \mu\right| \leq \frac{\lambda^{1/2}}{2\pi}\right\}.\]
    Then,
    \[\lim_{\lambda\to\infty} \frac{N \left(\lambda\right)}{\lambda^{n/2}} =\frac{\operatorname{vol}\left(T\right)}{ \left(4\pi\right)^{n/2}\Gamma\left(\frac{n}{2} + 1\right)}. \]
\end{theorem*}

As the above theorem implies that $N_b(\lambda) \in O(\lambda^d)$, we conclude that $N_b(\lambda)$ does not contribute to the leading coefficient asymptotics. 
Thus, we conclude that Theorem \ref{mainT} follows from Theorem \ref{thm:main} and Weyl's law for flat tori.


\subsection{Proof of Corollary \ref{mainC}}

The computations in Theorem \ref{thm:main} can be also used to obtain an analog of Weyl's law on $\left(p,q\right)$-forms since the action of $\square_b$ on $\left(0,q\right)$-forms is expressed diagonally by $\mathcal{L}_{d - 2q}$.The only technicality that remains is the multiplicity. Note that any computation for multiplicity for $\left(0,q\right)$-forms extends directly to $\left(p,q\right)$-forms by multiplication by $\binom{d}{p}$. 
If $\omega$ is a $q$-form, then it can be written as 
\[\sum_{|J|=q}\omega_J d\overline{z}^J,\]
where $\omega_J$ are functions, $J=\left(j_1, \ldots, j_q\right)$ with $1 \leq j_1  < \cdots <j_q \leq d$, and $d\overline{z}^J=d\overline{z}_1 \wedge \cdots \wedge d\overline{z}_q$. 

Noting that the $d\overline{z}^J$ are linearly independent, we have $\square_b \omega =\lambda \omega$ if and only if $\square_b \omega_J =\lambda \omega_J$ for each $J$.
From the convention $1 \leq j_1  < \cdots <j_q \leq d$, there are $\binom{d}{q}$ possibilities for $J$. Since $\square_b f=\lambda f$ implies  $\square_b f d\overline{z}^J=\lambda f d\overline{z}^J$, each eigenfunction of $\square_b$ induces $\binom{d}{q}$ many eigenforms. 

Therefore, to study $(0,q)$-forms, we set  $\alpha = d - 2q$ in the $- d < \alpha < d$ case of Theorem \ref{thm:main} and multiply the result by $\binom{d}{q}$. Then for $N_a(\lambda)$,

 \[\lim_{\lambda \rightarrow \infty} \frac{N_a(\lambda)}{\lambda^{d+1}}=\operatorname{vol} \left(M\right)\binom{d}{q} \frac{2}{\pi^{d + 1}\Gamma(d+2)} \int_{-\infty}^\infty \left(\frac{x}{\sinh x}\right)^d e^{- \left(d-2q\right) x} \,dx.\]
 
 Corollary \ref{mainC}
 follows immediately by multiplication by $\binom{d}{p}$ and noting that $N_b(\lambda)$ for $\square_b$ on $(p,q)$-forms is still in $O \left(\lambda^{d}\right)$.


This result is strikingly similar to the Weyl's law analog obtained by Stanton and Tartakoff in \cite{Stanton1984TheHE}. Note however, that their theorem requires that the manifold be an embedded hypersurface (that is co-dimension one). Though the Heisenberg group is such a manifold, we lose this property when passing to the quotient. The quotient is not a co-dimension one manifold. 
This difference is reflected in the difference by a factor of $2^{-d-2}$ in the leading coefficients. Furthermore, the result in \cite{Stanton1984TheHE} does not apply to functions, whereas our result in this note covers functions and differential forms of all degrees. 



\section*{Acknowledgements} 
First, we thank the other members of Team Hermann: Zoe Plzak, Ian Shors, and Samuel Sottile for their support during this work.
We also thank Kamryn Spinelli for his helpful comments on an earlier version of this paper. This research was completed at the REU Site: Mathematical Analysis and Applications at the University of Michigan-Dearborn. We would like to thank the National Science Foundation (DMS-1950102), the National Security Agency (H98230-21), the College of Arts, Sciences, and Letters, and the Department of Mathematics and Statistics for their support.


\newcommand{\etalchar}[1]{$^{#1}$}

\end{document}